\documentclass[a4paper,reqno]{amsart}
\usepackage[english]{babel}
\usepackage{amsmath,amsthm,amssymb,amsfonts}

\usepackage{soul}
\usepackage{comment}
\usepackage{hyperref}
\usepackage{mathtools}
\usepackage[T1]{fontenc}
\usepackage[utf8]{inputenc}
\usepackage{dsfont}
\hypersetup{
 colorlinks,
 linkcolor={blue!90!black},
 citecolor={red!80!black},
 urlcolor={blue!50!black}
}
\usepackage{tikz}
\usepackage{tikz-cd}
\usepackage{graphicx}
\usepackage[all,cmtip]{xy}
\usepackage{mleftright}
\usepackage{booktabs}
\usepackage{cite}
\usepackage{orcidlink}
\newtheorem{theorem}{\sc \textbf{Theorem}}[section]  
\newtheorem{proposition}[theorem]{\sc \textbf{Proposition}}   
\newtheorem{corollary}[theorem]{\sc \textbf{Corollary}}        
\newtheorem{lemma}[theorem]{\sc \textbf{Lemma}}

\theoremstyle{remark}
\newtheorem{definition}[theorem]{\sc \textbf{Definition}}

\newtheorem{remark}[theorem]{\sc \textbf{Remark}}

\def\cP{{\mathcal P}}

\def\cU{{\mathcal U}}
\def\cV{{\mathcal V}}

\numberwithin{equation}{section}

\def\namedlabel#1#2{\begingroup
    #2%
    \def\@currentlabel{#2}%
    \phantomsection\label{#1}\endgroup
}

\usepackage{pdfpages}

\makeatletter
\@namedef{subjclassname@1991}{Mathematics Subject Classification}
\usepackage{enumerate}
\graphicspath{{images/}}
\title[Harmonic Bergman spaces on locally finite trees]{Harmonic Bergman spaces on locally finite trees}
\author[A. Ottazzi]{Alessandro Ottazzi \orcidlink{0000-0002-4692-2751}}
\address[A. Ottazzi]{School of Mathematics and Statistics \\  University of New South
Wales \\ 2052 Sydney  \\  Australia}
\email{a.ottazzi@unsw.edu.au}
\author[F. Santagati]{Federico Santagati \orcidlink{0000-0003-4685-9956}}
\address[F. Santagati]{Dipartimento di Matematica and MaLGa center, \\ Università di Genova, Via Dodecaneso 35, Genova, Italy}
\email{federico.santagati@edu.unige.it}

\thanks{The authors were both supported by the Australian Research Council, grant DP220100285. Santagati was partially supported by the INdAM - GNAMPA Project ``$L^p$ estimates for singular integrals in nondoubling settings'' (CUP E53C23001670001).
}
\begin{document} 
\begin{abstract}
We define the harmonic Bergman space on locally finite trees with respect to a suitable probabilistic Laplacian and a class of weighted flow measures. We characterise the corresponding Bergman projection and prove that it is bounded on $L^p$ for every $p>1$, and of weak type $(1,1)$. We also prove necessary and sufficient conditions for the $L^p$-boundedness of the extension of a class of Toeplitz-type operators.
\end{abstract}
\maketitle
\section{Introduction} 

 Bergman spaces, i.e., spaces of holomorphic functions in $L^2$, have been studied in several geometric settings by different authors, from the classical unit disk setup to domains in $\mathbb C^n$ with boundaries of different complexities. The study has also been extended outside the realm of complex variable functions, where authors substitute holomorphic functions with harmonic functions with respect to a suitable Laplace operator. 
We briefly recall some classical results.

Let $D$ be the unit disk in $\mathbb{C}$, and let $\mu_\alpha$ be the absolutely continuous measure with respect to the Lebesgue measure in $D$ with density $(1 - |z|^2)^{\alpha-2}$ for $\alpha > 1$.  
Denote by $B^2(D, \mu_\alpha)$ the space of holomorphic functions in $L^2(D, \mu_\alpha)$. The Bergman projection $\cP_\alpha$, i.e., the orthogonal projection from $L^2(D, \mu_\alpha)$ onto $B^2(D, \mu_\alpha)$, is defined by
an integral operator, whose kernel is 
 the reproducing kernel of $B^2(D, \mu_\alpha)$.
 
A classical result states that the operator $\cP_\alpha$ is  bounded on $L^p(D, \mu_\alpha)$ for every $p \in (1,\infty)$;  we refer to \cite[Chapter 4]{Zhu} and the references therein for a more detailed discussion about $\cP_\alpha$.

The analysis on the disk can be naturally extended to the upper half-plane $U$, where the measure $\mu_\alpha$ is replaced by the absolutely continuous measure with respect to the Lebesgue measure on $U$, whose density is $\mathrm{Im} (z)^{\alpha-2}$, provided that $\alpha > 1$.

For harmonic Bergman functions with respect to the Lebesgue measure, we refer to \cite{Axler-Bourdon-Ramey} for the definition of the Bergman projection in the cases of the ball and the half-space,  to \cite{Ramey-Yi-TAMS} for the classical $L^p$-boundedness property on the half-space setting and to \cite{U,U2} where related results are extended to a wider class of measures. 

More recently, building on the connection between the disk and homogeneous trees \cite{CC, RT}, several works have extended the study of the Bergman projection to harmonic functions in the discrete setting; see \cite{CCPS,DMV,DMR}. In particular, in \cite{DMR} and \cite{DMV}, the authors study the $L^p$-boundedness of the Bergman projection on the homogeneous tree $T_{q+1}$, i.e., a tree where every vertex has exactly $q+1$ neighbours, where the harmonic functions are defined with respect to the \textit{combinatorial} Laplacian. In this context, the measures $(1 - |z|^2)^{\alpha-2} \, \mathrm{d}z$ and $\mathrm{Im} (z)^{\alpha-2} \, \mathrm{d}z$ are replaced by $q^{-\alpha |\cdot|}$ on the disk realization of $T_{q+1}$ and by $q^{\alpha \ell(\cdot)}$ on the upper half-plane realization of $T_{q+1}$, provided that $\alpha > 1$. Here, $|\cdot|$ denotes the graph distance from the origin, and $\ell(\cdot)$ is the \textit{level} of a vertex (see Section \ref{notation} for a precise definition). 
The study of this harmonic Bergman projection was recently generalised to the case of {\it radial} trees in \cite{CDMRV}.

It is worth emphasising that, unlike isotropic settings where no direction is privileged, the choice of a root on a tree naturally induces an intrinsic orientation, namely from the root toward the boundary. This directionality can be exploited by defining operators that are not isotropic but instead aligned with the geometry of the tree; such an approach often leads to improved analytical behavior. For instance, it is known that the centred Hardy--Littlewood maximal operator on a tree with bounded geometry may fail to be bounded on any $L^p$ space for $p<\infty$ \cite{LMSV}. However, if one restricts the averaging process to the portion of a ball lying \textit{below} its center, the resulting non-centered maximal operator becomes bounded on every $L^p$ space with $p > 1$, even on locally finite trees \cite{MS}. Similarly, while the combinatorial Laplacian, which is associated with an isotropic random walk, allows for the study of Carleson measures only under radial symmetry assumptions on the tree \cite{CCS2, CCS}, the use of a directional operator, such as the \textit{forward Laplacian}, defined as the identity minus an average over the {\it successors} of a vertex, captures the underlying geometry and enables a characterisation of Carleson measures on arbitrary locally finite trees \cite{OS}. 
This observation motivates the present work, where we study the harmonic Bergman projection associated to the forward Laplacian on locally finite trees, without imposing additional geometric assumptions. We briefly outline the structure of the article and the main results.

In Section \ref{Sec:2}, we find an explicit expression of the associated harmonic Bergman kernel on $L^2$ equipped with measures that are pointwise products of \textit{flow measures} and weights on $\mathbb{Z}$ satisfying suitable decay properties.
Notably, on the homogeneous tree $T_{q+1}$, the function $q^{\ell(\cdot)}$ represents the simplest example of a flow measure and is closely related to the measure $q^{\alpha\ell(\cdot)}$ considered in \cite{DMR}. The exponent $\alpha>1$ appearing in the measure $q^{\alpha \ell(\cdot)}$  can be interpreted as the product of $q^{\ell(\cdot)}$ and the weight on $\mathbb{Z}$ defined by $n \mapsto q^{n(\alpha-1)}$. The flow measure $q^{\ell(\cdot)}$ on $T_{q+1}$ can be regarded as the discrete counterpart of the measure $\mathrm{Im}(z)^{-1} \, \mathrm{d}z$ on the upper half-plane $U$, as implicitly noted in \cite{HS}. This parallel suggests that a flow measure alone is not enough to define a non-trivial Bergman space and an additional decay property of the measure as a vertex approaches the boundary is required to study the harmonic Bergman projection; indeed, we show that the Bergman space associated to a flow measure is trivial (see Proposition \ref{prop:harmonic}). 
We emphasise that the class of measures we consider includes, for instance, the radial measures studied in \cite{CDMRV}. Our main result, Theorem \ref{mainT}, establishes weak type (1,1) and $L^p$ boundedness of the Bergman projection associated to these measures for every $p \in (1,\infty)$. 
Our proof relies on the Calderón--Zygmund theory for nondoubling metric measure spaces, which is based on a dyadic approach, originally developed in \cite{LMP,CAP} (see also \cite{Tol}), and later extended to the discrete setting in \cite{CDMRV}.

Finally, in Section \ref{Sec:3},  we focus on the particular case where the flow measure is locally doubling and the weight we consider is an exponential function on $\mathbb Z$. In this case, we obtain sharp pointwise estimates for the Bergman kernel, and we prove necessary and sufficient conditions for the $L^p$-boundedness of the extension of a class of Toeplitz-type operators inspired by the operators considered in \cite{Zhu}.
If the tree is homogeneous and is endowed with the canonical flow measure, then the necessary and sufficient conditions coincide (Corollary~\ref{equivalence-Toeplitz}).

\subsection{Notation and preliminaries}\label{notation}
Throughout the paper, $C$ will denote a positive constant that may vary from line to line, independent of any involved variable, though it may depend on fixed parameters. If there exists a constant $C>0$ such that the positive quantities $A$ and $B$ satisfy $A/C \leq B \leq C A$, we write $A \approx B$.

Let $T$ be a tree and fix a reference vertex $o \in T$. We denote by $d$ the standard graph distance on $T$. We fix a root $\omega \in \partial T$, where $\partial T$ is the boundary of $T$, defined as the set of equivalence classes of infinite geodesics starting at $o$, where two geodesics are equivalent if they eventually coincide.
 We notice that the choice of a root in the boundary of $T$ induces a level function $\ell$ on $T$, which is the analogue of the Busemann function on a Riemannian manifold, and it is defined as 
\begin{align*}
    \ell(x):=\lim_{n \to \infty} n-d(x,x_n) \qquad \forall x \in T,
\end{align*} where $\{x_n\}_{n \in \mathbb N}$ is an enumeration of the ray $[o,\omega)$.  For every vertex $x \in T$ we set $p(x)$ as the unique neighbour of $x$  such that $\ell(p(x))=\ell(x)+1$. Similarly, we denote by $s(x)$ the set of neighbours of $x$ minus $p(x)$. Moreover, we inductively define  $$p^n(x)=p(p^{n-1}(x)) \qquad \forall n \ge 1, \forall x \in T$$ where we agree that $p^0(x):=x.$ Similarly, we set $$s_n(x)=\cup_{y \in s_{n-1}(x)} s(y) \qquad \forall n \ge 1, \forall x \in T,$$ where $s_0(x):=\{x\}.$ We assume that every vertex has at least one successor, so that $s(x) \ne \emptyset$ for every $x \in T$.
We say that a vertex $y$ lies above $x$ (or, equivalently, $x$ lies below $y$), if $y=p^n(x)$ for some $n \ge 0$. For every $x \in T$ we denote by $T_x$ the sector rooted at $x$, i.e., the set of vertices in $T$ lying below $x$. Precisely, 
$$T_x:= \bigcup_{n=0}^\infty s_n(x).$$
Similarly, we define the sector minus its vertex as $$\dot{T}_x:=\bigcup_{n=1}^{\infty} s_n(x)=T_x \setminus \{x\}.$$
For every pair of vertices $x$ and $y$ in $T$ we define their confluent $x\wedge y$ as the vertex of minimum level which lies above $x$ and $y$.
Given any measure $\mu$ on $T$ we define $L^p(\mu)$ as the space of $p$-summable functions on $T$ with respect to the measure $\mu$ and we denote by $\|\cdot\|_{L^p(\mu)}$ the associated usual $p$-norm. If $E \subset T$ we sometimes write $\|\cdot\|_{L^p(E,\mu)}$ to denote the $p$-norm on the measure space $(E,m)$.

 We define a flow measure as a positive function on $T$ satisfying the mass-preserving condition 
\begin{align}
   m(x)=\sum_{y \in s(x)}m(y) \qquad \forall x \in T.
\end{align}
Given a flow measure $m$, we set $w_{xy}:=m(y)/m(x)$ and $w_{xp(x)}:=0$ for every vertex $x \in T$ and $y \in s(x)$, and observe that $\sum_{y : d(x,y)=1} w_{xy}=1$. We associate to these weights a  Laplace operator $\Delta$, which we call weighted {\it forward} Laplacian, defined on a function $f$ by 
\begin{align}\label{def:Delta}
    \Delta f(x)=f(x)-\sum_{y:  d(x,y)=1} f(y) w_{xy}= f(x)-\sum_{y \in s(x)} f(y)\frac{m(y)}{m(x)}\qquad \forall x \in T.
\end{align}By \eqref{def:Delta}, it is clear that $\Delta$ is a {\it probabilistic} Laplacian in the sense that it is defined as the identity minus a weighted average on the set of neighbours of a vertex. As usual, a function $f$ on $T$ is said to be harmonic if $\Delta f=0$ on $T$.
In \cite{OS} it is shown that the Poisson integral associated with $\Delta$ can be used to characterise Carleson measures on locally finite trees.

%Observe that if $f$ is harmonic on $T$ then, obviously, the restriction of $f$ to $T_x$ is harmonic on $T_x$ for every $x \in T$, i.e, \begin{align}\label{restriction}
 %   f(y)= \frac{1}{m(y)}\sum_{z \in s(y)}f(x)m(z) \qquad \forall y \in T_x.
%\end{align}

%Define $\Delta_m=I-P_m$ where $P_{m}f(x)=\sum_{y \in s(x)} f(y) \frac{m_\nu(y)}{m_\nu(x)}.$ For simplicity set $m_\nu=m$. Let $g$ be a harmonic function on $\ell(x) \ge n$. We aim to extend $g$ to a harmonic function on $T$ radial on $T_y$ for every $y \in \ell_{n-1}$. Set 
%\begin{align*}
 %   g^H_{n}(x)=\begin{cases} g(x) &\ell(x) \ge n, \\ 
 %   g(y) &x \in T_y, \ell(y)=n-1,
  %  \end{cases}
%\end{align*} where $g$ on $\ell_{n-1}$ is such that $\sum_{y \in s(u)} g(y) \frac{m(y)}{m(u)}=g(u)$ for every $u \in \ell_n$. For every $u \in T$ define $$W_u=\bigg\{f: s(u) \to \mathbb C, \sum_{y \in s(u)} f(y) m(y)=0\bigg\},$$ and we set $$\langle f, g \rangle_{W_u}=\sum_{x \in s(u)} f(x) g(x) m(x).$$ 

\section{Bergman projection}\label{Sec:2} In this section, we provide an explicit formula for the harmonic Bergman kernel and then study the $L^p$ 
  boundedness properties of the associated Bergman projection on suitable weighted $L^p$ spaces that we now introduce.

Given a flow measure $m$ and a function $\sigma$ on $T$, we denote by $m \cdot \sigma$ the pointwise product of $m$ and $\sigma$. We will assume the following: \\ 
 \textbf{Assumption \label{assumption:i}i)}: $\sigma$ is a positive function on $T$ only depending on the level of a vertex, i.e.,  $$\sigma(x)=\sigma_0(\ell(x)),$$ for some $\sigma_0: \mathbb Z \to \mathbb R_+$. With a slight abuse of notation, we shall write $\sigma$ instead of $\sigma_0$ in the sequel. We also require that for every (or equivalently for one) integer $k \in \mathbb Z$  \begin{align*}
       \sum_{n \le k} \sigma(n) <\infty.
    \end{align*}  
   \textbf{Assumption \label{assumption:ii}ii)}: $\sigma(n)m(p^n(o))$ grows fast as $n \to \infty$, namely \begin{align}\label{assumptionii1}\sum_{n=0}^\infty \frac{1}{\sigma(n)m(p^n(o))}<\infty.\end{align} 

    Observe that \eqref{assumptionii1} implies that $L^2(m\cdot \sigma)$ satisfies the following vanishing property: \begin{align}\label{limite}\lim_{n \to \infty} f(p^n(o))=0 \qquad \forall f \in L^2(m \cdot \sigma). \end{align}
    Indeed,  since $\sigma(n)m(p^n(o))$ diverges as $n \to \infty$, there exists $N \in \mathbb N$ and $c>0$ such that  \begin{align}\label{assumtionii2}m(p^n(o))\sigma(n)>c \qquad \forall n>N. \end{align} 

%    Observe that $\eqref{assumtionii2}$ guarantees that $f \in L^2(m\cdot \sigma)$ satisfies the following vanishing property: 
 %   \begin{align}\label{limite}
  %      \lim_{n \to \infty} f(p^n(x))=0 \qquad \forall x\in T.
  %  \end{align}F

\begin{remark}\label{rem:carl}
    Notice that $$m\cdot \sigma(T_x)= \sum_{y \in T_x} m(y) \sigma(y)=\sum_{n=0}^\infty \sum_{y \in s_n(x)}m(y) \sigma(y)= m(x) \sum_{n \le \ell(x)} \sigma(n) \qquad \forall x \in T.$$ This implies that $m \cdot \sigma$ is a {\it Carleson} measure in the sense of \cite[Definition 2.6]{OS} if  there exists $c>0$ such that  $$m\cdot \sigma (T_x)\le c m(x) \qquad \forall x \in T$$ that is equivalent to \begin{align}\label{prop:carl}\sum_{n \le k} \sigma(n) \le c \qquad \forall k \in \mathbb Z,\end{align} namely, $\sigma \in \ell^1(\mathbb Z)$. 
    %\st{It is not hard to see that \eqref{prop:carl} holds if and only if $\sigma(n) \approx a^n$ for some $a>1$.}
\end{remark}  

\begin{definition}
We define the Bergman space $B^2(m\cdot\sigma)$ as the space of $L^2(m\cdot \sigma)$ harmonic functions equipped with the $L^2(m\cdot \sigma)$ inner product.
\end{definition}

We recall the following property of harmonic functions; see \cite[Eq. (2.1)]{OS} for a complete proof of this fact. 
\begin{lemma}\label{lem:harm}
    Let $f$ be a harmonic function. Then, 
    \begin{align*}
        \sum_{y \in s_n(x)} f(y) m(y)= f(x)m(x) \qquad \forall n \in \mathbb N, \forall x \in T.
    \end{align*}
\end{lemma}

We note that the above lemma implies that there are no nontrivial harmonic functions in $L^p(m)$. In fact more is true.

\begin{proposition}\label{prop:harmonic} Assume that $\sum_{n \le 0} \tilde{\sigma}(n)=\infty$ for some positive weight $\tilde{\sigma}$ on $\mathbb Z$.
    Let $f$ be harmonic on $T$ and suppose that $f \in L^p(m\cdot \tilde{\sigma})$ for some $p \in [1,\infty)$. Then, $f=0$ on $T$.
\end{proposition}
\begin{proof}
  Let $f$ be harmonic  and suppose that $f \in L^p(m\cdot \tilde{\sigma})$. By contradiction, assume that $f(x) \ne 0$ for some $x \in T$. Then, by Lemma \ref{lem:harm}
  \begin{align*}
      |f(x)|^p=\bigg|\sum_{y \in s_n(x)} f(y) \frac{m(y)}{m(x)}\bigg|^p \le \frac{1}{m(x)}\sum_{y \in s_n(x)}|f(y)|^p m(y) \qquad \forall n \in \mathbb N,
  \end{align*} where we have used that $m$ is a flow measure so $\sum_{y \in s_n(x)}m(y)=m(x)$ and Jensen's inequality in the last step. It follows that 
  \begin{align*}
  \|f\|_{L^p(T_x,m\cdot \tilde{\sigma})}^p&=\sum_{n=0}^\infty \sum_{y \in s_n(x)}|f(y)|^p m(y)\tilde{\sigma}(\ell(y)) \\&=\sum_{n=0}^\infty \tilde{\sigma}(\ell(x)-n) \sum_{y \in s_n(x)}|f(y)|^p m(y) \\  %&\ge |f(x)|^p\sum_{n=0}^\infty \sum_{y \in s_n(x)} m(x)\tilde{\sigma}(\ell(y)) \\ &\ge |f(x)|^p\sum_{n=0}^\infty \sum_{y \in s_n(x)} m(y)\tilde{\sigma}(\ell(y))\\ 
  %&=|f(x)|^p\sum_{n=0}^\infty \tilde{\sigma}(\ell(x)-n)\sum_{y \in s_n(x)} m(y)\\
  &\ge|f(x)|^pm(x)\sum_{n=0}^\infty \tilde{\sigma}(\ell(x)-n) \\&=\infty,
  \end{align*} which is a contradiction because $f \in L^p(m\cdot \tilde{\sigma})$.
\end{proof}

For every $x \in T$ and $n \in \mathbb Z$ we set 
\begin{align}\label{cn}
S_{n,x}(\sigma)=\sum_{k=0}^\infty \sigma(\ell(x)+n-k).
\end{align}
We shall write $S_x(\sigma)$ in place of $S_{0,x}(\sigma)$. Notice also that $S_{n,x}(\sigma)=S_{p^n(x)}(\sigma)$ for every $n \in \mathbb N$ and $x \in T$. 

The following function will be useful for obtaining an explicit expression of the Bergman kernel. We define $\Psi: T \times T \times T \to \mathbb R$ by \begin{align}\label{Psi}
 \Psi(v,z,x):=\begin{cases} 0 &\{x,z\}\not\subset \dot{T}_v, \\ \frac{1}{m(y)}-\frac{1}{m(v)} 
    &\{x,z\}\subset T_y, y \in s(v), \\ 
    -\frac{1}{m(v)} &\text{otherwise}.
    \end{cases}
\end{align} We point out that the support of $\Psi(p^{n+1}(x),p^n(x),\cdot)$ is exactly $\dot{T}_{p^{n+1}(x)}$ for every $n \in \mathbb N$ and $x \in T$.

To simplify the notation, we introduce the gradient $\nabla f(x) := f(x) - f(p(x))$, which appears in the statement of the next lemma.
\begin{lemma}
    The following holds 
    \begin{align}\label{f: fgnH} \langle f, \Psi(p(x),x,\cdot) \rangle_{L^2(m\cdot\sigma)}=S_x(\sigma) \nabla f(x) \qquad \forall f \in B^2(m \cdot \sigma), \forall x \in T.
 %   \langle f, g_n^H \rangle_{B^2(m\cdot\sigma)}=c_n\sum_{x \in s(y)} f(x) g(x) m(x),
\end{align} 
\end{lemma}
\begin{proof}
    Fix $x \in T$. Then,
    \begin{align*}
        \langle f, \Psi(p(x),x,\cdot) \rangle_{L^2(m\cdot\sigma)}&=\sum_{y \in T} f(y)\Psi(p(x),x,y) m(y) \sigma(y) \\ 
        &=\sum_{y \in \dot{T}_{p(x)}} f(y)\Psi(p(x),x,y) m(y)\sigma(y)\\ 
        &=\sum_{n=0}^\infty \sigma(\ell(p(x))-n-1)\sum_{u \in s(p(x))}\sum_{y \in s_{n}(u)} f(y) \Psi(p(x),x,y) m(y).
        \end{align*}By \eqref{Psi} and Lemma \ref{lem:harm}, we get that the inner sums above are equal to 
        \begin{align*}
       &\sum_{u \in s(p(x))\setminus \{x\}}\sum_{y \in s_{n}(u)} \frac{-f(y)}{m(p(x))}  m(y)+ \sum_{y \in s_n(x)}f(y)\bigg(\frac{1}{m(x)}-\frac{1}{m(p(x))}\bigg)m(y) \\ 
       &=\sum_{u \in s(p(x)) \setminus\{x\}}-f(u)\frac{m(u)}{m(p(x))}+f(x)\bigg(1-\frac{m(x)}{m(p(x))}\bigg) \\ 
       &=f(x)- \sum_{y\in s(p(x))} f(y) \frac{m(y)}{m(p(x))}\\ 
       &=f(x)-f(p(x))
    \\   &=\nabla f(x).
    \end{align*} The result follows by noticing that $$\sum_{n=0}^\infty \sigma(\ell(p(x))-n-1)=\sum_{n=0}^\infty \sigma(\ell(x)-n)=S_x(\sigma).$$
\end{proof}

\begin{proposition}
    $B^2(m\cdot\sigma)$ is a Hilbert space.
\end{proposition}
\begin{proof}
%\st{It suffices to prove that $\Delta$ is bounded on  $L^2(m\cdot \sigma)$. Indeed, in this case, it would follow that $B^2(m\cdot\sigma)$ is closed in $L^2(m\cdot \sigma).$ Let $f \in B^2(m\cdot\sigma)$. Then,} 

   % \begin{align*}
    %\sum_{x \in T} |P_m f(x)|^2m(x)\sigma(x) &\le  \sum_{x \in T} \sum_{y \in s(x)} |f(y)|^2 \frac{m(y)}{m(x)} m(x)\sigma(x) \\ 
    %    &= \sum_{y \in T} |f(y)|^2  m(y) \sum_{x=p(y)}  \sigma(x) \\ 
     %   &\le C \sum_{y \in T} |f(y)|^2  m(y)\sigma(y),
    %\end{align*} where in the last inequality we have used $iii)$. Since $\Delta_m=I-P_m$, this concludes the proof.
It suffices to prove that $B^2(m\cdot\sigma)$ is closed in $L^2(m \cdot \sigma)$.
    Let $\{f_n\}_{n \in \mathbb N} \subset B^2(m\cdot\sigma)$ be a convergent sequence in $L^2(m\cdot \sigma)$ and denote its limit by $f$. Then, there exists a subsequence $\{f_{n_k}\}_{k \in \mathbb N}$ that converges {\it pointwise} to $f$ (this follows because there are no non-trivial subsets of $T$ of measure $0$). It follows that for every $x \in T$ 
    $$0=\Delta f_{n_k}(x)= f_{n_k}(x)-\sum_{y \in s(x)} f_{n_k}(y) \frac{m(y)}{m(x)} \xrightarrow[\mathrm{as} \ k \to \infty]{} f(x)-\sum_{y \in s(x)} f(y) \frac{m(y)}{m(x)}=\Delta f(x),$$
 because $s(x)$ contains a finite number of vertices.
\end{proof}
We now introduce the function $K : T \times T \to \mathbb C$ defined by 
\begin{align}\label{def:K}
    K(x,y)=\sum_{n=0}^\infty \frac{\Psi(p^{n+1}(x),p^n(x),y)}{S_{n,x}(\sigma)}.
\end{align}
We shall prove that $K(\cdot,\cdot)$ is the reproducing kernel of $B^2(m\cdot\sigma)$.
\begin{remark}\label{rem:Psi} Notice that $K$ satisfies the following symmetry property: 
\begin{align*}
   K(x,y)=K(y,x) \qquad \forall n \in \mathbb N, \forall x,y \in T.
\end{align*} %This follows by observing that $K(\cdot, \cdot)$ is real, hence for every couple of point $x,y \in T$
%\begin{align*}
 %   K(x,y)=\langle K(x,\cdot), K(y,\cdot) \rangle_{B^2(m\cdot \sigma)}=\langle K(y,\cdot), K(x,\cdot) \rangle_{B^2(m\cdot \sigma)}=K(y,x).
%\end{align*}

Indeed, recall that
$\Psi(p^{n+1}(x),p^n(x),y)$ is not zero if and only if $y \in \dot{T}_{p^{n+1}(x)}.$ We distinguish two cases: if $x,y$ are not one above the other, then, $y \in \dot{T}_{p^{n+1}(x)}$ implies $x \wedge y \in T_{p^{n+1}(x)}$, thus by \eqref{Psi} 
\begin{align}\label{xnotabovey}
  \nonumber  K(x,y)&=\sum_{n: x\wedge y \in {T}_{p^{n+1}(x)}} \frac{\Psi(p^{n+1}(x),p^n(x),y)}{S_{n,x}(\sigma)}\\ &=\sum_{n \ge 0}\frac{1}{S_{n,x\wedge y}(\sigma)}\bigg(\frac{1}{m(p^{n}(x\wedge y))}-\frac{1}{m(p^{n+1}(x \wedge y))}\bigg)\\ \nonumber&-\frac{1}{S_{p^{n_0}(x)}(\sigma)m(x\wedge y)},
\end{align} where $n_0$ is such that $p^{n_0+1}(x)=x\wedge y$. Observe that $$S_{p^{n_0}(x)}(\sigma)=\sum_{k \ge 0} \sigma(\ell(x\wedge y)-1-k),$$ so that the above expression of $K$ is symmetric in $x$ and $y$ because $x\wedge y=y \wedge x$.

If $x$ and $y$ are one above the other, then $x\wedge y \in \{x,y\}$. In this case, we have that
\begin{align*}
    K(x,y)&=\sum_{n \ge 0}\frac{1}{S_{n,x\wedge y}(\sigma)}\bigg(\frac{1}{m(p^{n}(x\wedge y))}-\frac{1}{m(p^{n+1}(x \wedge y))}\bigg),
\end{align*}  which is again symmetric in the two variables. 
\end{remark} 

The next lemma will be instrumental to prove that $K$ is the reproducing kernel of $B^2(m \cdot \sigma)$. 
\begin{lemma}\label{lem:ort}
    The following hold: 
    \begin{itemize}
        \item[a)] for every $x,y \in T$ \begin{align}\label{claimPsi}
    \sum_{n=0}^\infty \bigg| \frac{\Psi(p^{n+1}(x),p^n(x),y)}{S_{n,x}(\sigma)}\bigg|<\infty;
    \end{align}
    \item[b)] for every $v\in T$ the family $\{\Psi(p^{k+1}(v),p^k(v),\cdot)\}_{k \in \mathbb N}$ forms an orthogonal system in $B^2(m \cdot \sigma)$. 
    \end{itemize}
\end{lemma}
\begin{proof} %We claim that for every $x,y \in T$ \begin{align}\label{claimPsi}
 %   \sum_{n=0}^\infty \bigg| \frac{\Psi(p^{n+1}(x),p^n(x),y)}{S_{n,x}(\sigma)}\bigg|<\infty.
  %  \end{align} 
  We first prove a). Indeed, define$$C_{x,y}=\begin{cases}0 &\text{if $x$ and $y$ are one above the other,} \\ 
  \frac{1}{S_{n_0,x}(\sigma)m(x\wedge y)} &\text{otherwise,}      
  \end{cases}$$ where in the second case $n_0$ is the unique integer such that $p^{n_0+1}(x)=x\wedge y$.
  By arguing as in Remark \ref{rem:Psi}
    \begin{align*}
         \sum_{n=0}^\infty \bigg| \frac{\Psi(p^{n+1}(x),p^n(x),y)}{S_{n,x}(\sigma)}\bigg|&\le C_{x,y}+ \sum_{n \ge 0} \frac{1}{m(p^n(x\wedge y))S_{n,x \wedge y}(\sigma)} \\ &\le C_{x,y}+\sum_{n \ge 0} \frac{1}{m(p^n(x\wedge y))\sigma(p^n(x\wedge y))}\\ &< \infty,
    \end{align*} by Assumption \hyperref[assumption:ii]{ii)}.
    
    We next prove b). Assume without loss of generality that $j > k$. We have that 
\begin{align}\nonumber\big\langle &\Psi(p^{k+1}(v),p^k(v), \cdot), \Psi(p^{j+1}(v),p^j(v),\cdot)\big\rangle_{L^2(m\cdot \sigma)}  \\ \nonumber &=\sum_{y \in \dot{T}_{p^{j+1}(v)}\cap \dot{T}_{p^{k+1}(v)}} \Psi(p^{k+1}(v),p^k(v), y)\Psi(p^{j+1}(v),p^j(v),y)m(y)\sigma(y) \\ \label{scalarprod}
&=\sum_{y \in \dot{T}_{p^{k+1}(v)}} \Psi(p^{k+1}(v),p^k(v), y) \Psi(p^{j+1}(v),p^j(v),y) m(y)\sigma(y).
\end{align} Observe that $\Psi(p^{j+1}(v),p^j(v),\cdot)$ is constant on $\dot{T}_{p^{k+1}(v)}$, so \eqref{scalarprod} is equal to
\begin{align*}
&M_{j,v}\bigg[\sum_{y \in \dot{T}_{p^{k+1}(v)} \setminus T_{p^k(v)}}\!\!\!\!\!\! -\frac{m(y)\sigma(y)}{m(p^{k+1}(v))} +\sum_{y \in T_{p^k(v)}} \bigg(\frac{1}{m(p^{k}(v))}-\frac{1}{m(p^{k+1}(v))}\bigg)m(y)\sigma(y)\bigg] \\ 
&=M_{j,v}S_{k,v}(\sigma)\bigg[\frac{m(p^{k}(v))-m(p^{k+1}(v))}{m(p^{k+1}(v))} + \bigg(\frac{1}{m(p^k(v))}-\frac{1}{m(p^{k+1}(v))}\bigg)m(p^k(v))\bigg] \\ 
&=0,
\end{align*} where  $M_{j,v}:=\bigg(\frac{1}{m(p^{j}(v))}-\frac{1}{m(p^{j+1}(v))}\bigg)$.
 \end{proof}

%  \textcolor{red}{In the next proposition we prove that $K(v,\cdot)$ belongs to $B^2(m \cdot \sigma)$ for every $v \in T$.}
\begin{proposition}\label{prop:convergence}
    For every $v \in T$, $K(v,\cdot) \in B^2(m \cdot \sigma).$
\end{proposition}
\begin{proof} We first show that $K(v,\cdot)$ is harmonic for every $v \in T$.  %A similar estimate can be proved when $x$ and $y$ are one above the other. 
    Notice that
\begin{align}\label{mediaPsi}
    \sum_{z \in s(y)}\Psi(p^{n+1}(x),p^n(x),z)m(z)=\begin{cases} 0 &\text{if $y \not \in \dot{T}_{p^{n+1}(x)}$,} \\ 
    \bigg(\frac{1}{m(p^n(x))}-\frac{1}{m(p^{n+1}(x))}\bigg)m(y)&\text{if $y \in T_{p^n(x)}$,} \\ 
    -\frac{m(y)}{m(p^{n+1}(x))} &\text{otherwise.}
    \end{cases}
\end{align}
Set $K_v(\cdot)=K(v,\cdot)$ and observe that by \eqref{claimPsi} 
\begin{align*}
\sum_{z \in s(y)} K_v(z)\frac{m(z)}{m(y)}&=\sum_{n \ge 0}\sum_{z \in s(y)} \frac{\Psi(p^{n+1}(v),p^n(v),z)}{S_{n,v}(\sigma)}\frac{m(z)}{m(y)}=:I_0. 
\end{align*}
Now we have two possibilities: either $v$ and $y$ are not one below the other, or they are. In the first case, set $n_0$ as the integer such that $p^{n_0+1}(v)=v\wedge y$ as above and observe that an application of \eqref{mediaPsi} and \eqref{xnotabovey} yields
\begin{align*}
I_0&=\sum_{n\ge n_0+1} \frac{1}{S_{n,v}(\sigma)}\bigg(\frac{1}{m(p^n(v))}-\frac{1}{m(p^{n+1}(v))}\bigg)- \frac{1}{S_{n_0,v}(\sigma)m(p^{n_0+1}(v))}\\ 
&=K_v(y),
\end{align*} that is equivalent to say that $\Delta K_v(y)=0$. 
The case when $x$ and $y$ are one below the other is similar and easier; we omit the details. 

It remains to prove that $K_v \in L^2(m \cdot \sigma)$ for every $v \in T$.
    We show that $$\bigg\|K(v,\cdot)-\sum_{j=0}^N\frac{ \Psi(p^{j+1}(v),p^j(v),\cdot)}{S_{j,v}(\sigma)}\bigg\|_{L^2(m\cdot \sigma)}\xrightarrow[\text{as} \  N \to \infty]{} 0.$$ Observe that by Lemma \ref{lem:ort} b), for every couple of integers $N\le M$, \begin{align}\label{triangle}
\bigg\|\sum_{j=N}^M\frac{\Psi(p^{j+1}(v),p^{j}(v),\cdot)}{S_{j,v}(\sigma)}\bigg\|_{L^2(m\cdot \sigma)}^2\le \sum_{j=N}^\infty \bigg\|\frac{\Psi(p^{j+1}(v),p^{j}(v),\cdot)}{S_{j,v}(\sigma)}\bigg\|_{L^2(m\cdot \sigma)}^2.
\end{align} We next estimate each summand of the above sum. Observe that
\begin{align*}
    \bigg\|\frac{\Psi(p^{j+1}(v),p^{j}(v),\cdot)}{S_{j,v}(\sigma)}\bigg\|_{L^2(m\cdot \sigma)}^2&=\sum_{x \in T_{p^{j+1}(v)}} \frac{\Psi(p^{j+1}(v),p^j(v),x)^2}{S_{j,v}(\sigma)^2} m(x)\sigma(x) 
    \end{align*} that by \eqref{Psi} is in turn equal to
    \begin{align*}
    \frac{1}{S_{j,v}(\sigma)^2}\bigg[\bigg(\frac{1}{m(p^{j}(v))}&-\frac{1}{m(p^{j+1}(v))}\bigg)^2\sum_{x \in T_{p^j(v)}}m(x)\sigma(x)+ \\ &+\frac{1}{m(p^{j+1}(v))^2}\sum_{y \in s(p^{j+1}(v)) \setminus \{p^j(v)\}}\sum_{x \in T_y}m(x)\sigma(x)\bigg]\\ 
    &=:I+II.
\end{align*} By a direct computation 
\begin{align*}
    I&=\frac{m(p^j(v))}{S_{j,v}(\sigma)}\bigg(\frac{1}{m(p^{j}(v))}-\frac{1}{m(p^{j+1}(v))}\bigg)^2 \\ II&=\frac{1}{S_{j,v}(\sigma)m(p^{j+1}(v))^2}\big(m(p^{j+1}(v))-m(p^{j}(v))\big),
\end{align*} so that $$I+II \le \frac{C}{S_{j,v}(\sigma)m(p^{j}(v))}.$$It follows by \eqref{triangle} that 
\begin{align*}
    \bigg\|\sum_{j=N}^M\frac{\Psi(p^{j+1}(v),p^{j}(v),\cdot)}{S_{j,v}(\sigma)}\bigg\|_{L^2(m\cdot \sigma)}^2\le C \sum_{j=N}^\infty\frac{1}{S_{j,v}(\sigma)m(p^j(v))},
\end{align*} and right-hand side tends to zero when $N \to \infty$ by Assumption \hyperref[assumption:ii]{ii)} because $p^j(v) \in [o,\omega)$ for every $j \ge N_0$ with $N_0$ sufficiently large.
\end{proof}
\begin{proposition}
    $K(v,\cdot)$ is the reproducing kernel of $B^2(m\cdot \sigma)$.
\end{proposition} \begin{proof}
    By \eqref{f: fgnH} applied to $x=p^n(v)$, $$\langle f, \Psi(p^{n+1}(v),p^n(v),\cdot)\rangle_{L^2(m\cdot\sigma)}=S_{n,v}(\sigma) \nabla f(p^n(v)).$$  
    We conclude that 
    \begin{align}
        \langle f, K(v,\cdot)\rangle_{L^2(m \cdot \sigma)}=\sum_{n=0}^\infty f(p^{n}(v))-f(p^{n+1}(v))=f(v),
    \end{align} where in the last equality we have used Proposition \ref{prop:convergence}  and \eqref{limite}.
\end{proof}

We  define the Bergman projection $\cP$ as the operator acting on $L^2(m \cdot \sigma)$ functions by
\begin{align*}
    \cP f(x) = \sum_{y \in T} K(x,y) f(y) m(y) \sigma(y) \qquad \forall x \in T.
\end{align*}
Notice that $\cP f(x)= \langle f, K(x,\cdot) \rangle_{L^2(m \cdot \sigma)}$ for every $x \in T$. It is easy to see that $\cP$ is an orthogonal projection on $L^2(m \cdot \sigma)$. Indeed, by Remark \ref{rem:Psi}, $\cP$ is self-adjoint on $L^2(m \cdot \sigma)$ and for every $f \in L^2(m \cdot \sigma)$
\begin{align*}
    \Delta \cP f(x)&= \cP f(x)-\sum_{z \in s(x)} \cP f(z) \frac{m(z)}{m(x)}  \\ 
    &=\sum_{y \in T} K(x,y) f(y) m(y) \sigma(y)-\sum_{z \in s(x)}\sum_{y \in T} K(z,y) f(y) m(y)\frac{m(z)}{m(x)} \sigma(y) \\ 
    &=\sum_{y \in T} \Delta K(\cdot, y)(x) f(y) m(y) \sigma(y) \\ 
    &=0,
\end{align*}because $K$ is symmetric and $K_y$ is harmonic. If $f \in B^2(m \cdot \sigma)$ then $$\cP f(x)=\langle f, K(x,\cdot)\rangle_{L^2(m \cdot \sigma)}=f(x) \qquad \forall x \in T,$$ so $\cP$ is an orthogonal projection.
We now discuss the $L^p(m\cdot \sigma)$ boundedness of $\cP$ for $p \in [1,\infty)$.
\begin{theorem}\label{mainT} $\cP$ is bounded on $L^p(m\cdot \sigma)$ for every $p \in (1,\infty)$ and it is bounded from $L^1(m \cdot \sigma)$ to $L^{1,\infty}(m \cdot \sigma).$

\end{theorem}
\begin{proof}
   It suffices to prove the weak type $(1,1)$ boundedness of $\cP$, then the result follows by interpolation with the trivial $L^2(m \cdot \sigma)$ boundedness of $\cP$ and by duality.

   In order to prove that $\cP : L^1(m \cdot \sigma) \to L^{1,\infty}(m \cdot \sigma)$ we shall make use of \cite[Theorem C]{CDMRV}. We first show that $K(\cdot,\cdot)$ satisfies a suitable H\"ormander integral condition. We shall prove that 
    \begin{align}\label{hormander}
        \sup_{u \in T} \sup_{x,y \in T_u} \sum_{z \in T \setminus T_u} |K(z,x)-K(z,y)| \ \sigma(z)m(z)<\infty.
    \end{align}
Indeed, it suffices to notice that if $x \in T_u$ for some $u \in T$, then for every $z \in T \setminus T_u$ 
    \begin{align*}
      \Psi(p^{n+1}(z),p^n(z),x)=\begin{cases}
            \frac{1}{m(p^n(z))}-\frac{1}{m(p^{n+1}(z))}&\text{if $x \in T_{p^n(z)}$,} \\ 
            -\frac{1}{m(p^{n+1}(z))} &\text{$x\in \dot{T}_{p^{n+1}(z)} \setminus T_{p^n(z)}$,} \\ 
            0 &\text{otherwise.}
        \end{cases} 
    \end{align*} Now we make the following observation: if $x,y \in T_u$ and $z \in T \setminus T_u$, then \begin{align*}x \in  T_{p^{n}(z)} &\iff y \in T_{p^{n}(z)}, \\ 
    x \in  \dot{T}_{p^{n+1}(z)} \setminus T_{p^{n}(z)} &\iff  y \in  \dot{T}_{p^{n+1}(z)} \setminus T_{p^{n}(z)}.
    \end{align*}Therefore it follows that for such a choice of $x,y$ and $z$ $$\Psi(p^{n+1}(z),p^n(z),x)=\Psi(p^{n+1}(z),p^n(z),y)$$ and hence $K(z,x)=K(z,y)$. This implies that H\"ormander's condition is trivially satisfied.

Next, we shall show that  the following {\it size} condition holds
    \begin{align}\label{size}
        \sup_{u \in T} \sup_{x \in T_u} \sum_{z \in T_{p(u)}\setminus T_u} |K(x,z)| \, m(z)\sigma(z)<\infty.
\end{align}
    Fix $u \in T$ and $x \in T_u$. Then, define $n_0$ as the integer such that $p^{n_0+1}(x)=p(u)$ and notice that for every $z \in \dot{T}_{p(u)} \setminus T_u$ we have that 
    \begin{align*}|K(x,z)| &\le \sum_{n=0} \frac{|\Psi(p^{n+1}(x),p^n(x),z)|}{S_{n,x}(\sigma)} \\  &\le  \sum_{\substack{n \ge n_0+1}} \frac{1}{S_{n,x}(\sigma)}\bigg(\frac{1}{m(p^n(x))}-\frac{1}{m(p^{n+1}(x))} \bigg)+\frac{1}{S_{n_0,x}(\sigma)m(p(u))} \\ 
&\le \sum_{\substack{n \ge n_0+1}} \bigg(\frac{1}{S_{n,x}(\sigma)m(p^n(x))}-\frac{1}{S_{n+1,x}(\sigma)m(p^{n+1}(x))} \bigg)+\frac{1}{S_{n_0,x}(\sigma)m(p(u))}
    \\
    &\le \frac{2}{S_{n_0,x}(\sigma)} \frac{1}{m(p(u))},
    \end{align*} where we have used that $ \frac{1}{S_{n,x}(\sigma)m(p^n(x))} \xrightarrow[\mathrm{as} \ n \to \infty]{} 0$ and that $$S_{n,x}(\sigma) \le S_{n+1,x}(\sigma) \qquad \forall x \in T, \forall n \in \mathbb Z.$$ Similarly, it is not hard to see that $$|K(x,p(u))|\le \frac{1}{S_{p(u)}(\sigma)m(p(u))}.$$ It follows that 
    \begin{align*}
       \sum_{z \in T_{p(u)}\setminus T_u} |K(x,z)| \, m(z)\sigma(z) &\le 1+\sum_{z \in \dot{T}_{p(u)}} \frac{2}{S_{n_0,x}(\sigma)} \frac{m(z)\sigma(z)}{m(p(u))} \\ 
       &=1+\sum_{ \ell\le \ell(u)} \frac{2}{S_{n_0,x}(\sigma)} \frac{m(p(u))\sigma(\ell)}{m(p(u))} \\ 
       &\le 3,
    \end{align*} because $S_{n_0,x}(\sigma)=\sum_{\ell \le \ell(u)} \sigma(\ell)$.  This proves \eqref{size} holds and the result now follows by invoking \cite[Theorem C]{CDMRV}.
\end{proof}
\begin{remark}
 Theorem \ref{mainT}   shows that $\cP : H^1(T) \to L^1(T)$ and $\cP : L^\infty(T) \to BMO(T)$, where $H^1(T)$ and $BMO(T)$ are the martingale Hardy and BMO adapted to the natural filtration induced by the Gromov metric, see \cite[Section 1]{CDMRV}.
\end{remark}
\begin{remark}[Rooted trees]
In this paper, we consider trees with the root at infinity, whereas in the literature, \textit{rooted} trees (i.e., trees with a vertex as the root) are often studied. We briefly explain how our results can be easily adapted to this different case. Assume that $T_o$ is a rooted tree, $m$ is a flow measure on $T_o$ and $\sigma \in \ell^1(\mathbb N)$. A key difference between rooted trees and trees with root at infinity is that in the former case, the space $B^2(m \cdot \sigma)$  contains constant functions as $m \cdot \sigma$ is finite on $T_o$; this implies that the structure of the Bergman projection is slightly different. It is straightforward to verify that the function $K_o : T_o \times T_o \to \mathbb C$ defined by  $$K_o(x,y)=\mathfrak{K}_o+\sum_{n=0}^{|x|-1} \frac{\Psi(p^{n+1}(x),p^n(x),y)}{S_{n,x}(\sigma)}$$ is the kernel of the Bergman projection in this setting; here $\mathfrak{K}_o=\frac{1}{m\cdot \sigma(T_o)}$,  $|x|=d(x,o)$, $\Psi$ is defined as in \eqref{Psi}, and $S_{n,x}(\sigma)$ as in \eqref{cn}, and where the sum above is zero when $x=o$. Also, note that in contrast to what happens in the case of trees with root at infinity, no additional assumptions on $m$ are required.
Notice that when $T_o$ is a radial rooted tree such that $x \in T_o$ has $q(|x|)+1$ neighbours for some $q : \mathbb N \to  \mathbb N$ with $q \ge 2$, then the function $m_c(x):=\prod_{\ell=0}^{|x|-1}q(\ell)^{-1}$ is a flow measure, indeed 
\begin{align*}
    \sum_{y \in s(x) }m_c(y)=q(|x|) \prod_{\ell=0}^{|x|}q(\ell)^{-1}=m_c(x) \qquad \forall x \in T_o.
\end{align*} We call such $m_c$ {\it canonical} flow measure. 
Thus if for some $\alpha>1$ we define $$\sigma_\alpha(n):=\prod_{\ell=0}^{n-1}q(\ell)^{-\alpha}\bigg(\prod_{\ell=0}^{n-1}q(\ell)^{-1}\bigg)^{-1}=\prod_{\ell=0}^{n-1} q^{1-\alpha}(\ell) \le 2^{n(1-\alpha)},$$
then $\sum_{n=0}^\infty \sigma_\alpha(n)<\infty$. With a slight abuse of notation, we denote by $\sigma_\alpha$ the function on $T$ defined by $\sigma_\alpha(x)=\sigma_\alpha(|x|)$. The measure $\mu_\alpha=m_c \cdot \sigma_\alpha$ is the one considered in \cite{CDMRV}. 
\end{remark}
\section{Exponential measures and integral operators}\label{Sec:3}
Let $T$ denote a tree with root at infinity. In analogy with the analysis performed in \cite[Chapter 3.4]{Zhu} and in \cite[Section 3]{DMV}, we shall study the weighted $L^p$-boundedness of integral operators whose kernel is given by the product of a weight and the Bergman kernel.

We define the two classes of operators
\begin{align*}
    \cU_{a,b,c} f(x)&=m(x)\sigma_a(x)\sum_{y \in T} K_c(x,y) f(y)m(y)\sigma_b(y),\\
    \cV_{a,b,c} f(x)&=m(x)\sigma_a(x)\sum_{y \in T} |K_c(x,y)| f(y)m(y)\sigma_b(y),
\end{align*}
where $\sigma_a, \sigma_b$, and $\sigma_c$ are weights on $\mathbb Z$ and $K_c$ denotes the kernel of the Bergman projection on $L^2(m \cdot \sigma_c)$. We assume that $m$ and $\sigma_c$ satisfy Assumptions \hyperref[assumption:i]{i)} and \hyperref[assumption:ii]{ii)}, ensuring that the theory developed in the previous section applies to $K_c$. We are interested in the study of the $L^p(m\cdot \sigma_d)$ boundedness of such operators in terms of a weight $\sigma_d$ on $\mathbb Z$. 

For this purpose, we need precise estimates of $K_c$ that we obtain under some additional assumptions. From now on, we require that every vertex in $T$ has at least two successors and that $m$ satisfies the following condition
\begin{align}\label{locdoub}
    \sup_{x \in T} \frac{m(p(x))}{m(x)}=:C_m<\infty.
\end{align} 
Notice that this assumption is equivalent to the fact that the metric measure space $(T,d,m)$ is locally doubling \cite[Proposition 2.2]{LSTV}  and implies that the number of neighbours of a vertex is bounded on $T$; here $d$ is the usual graph distance on $T$. It is known \cite[Proposition 2.2]{LSTV} that \eqref{locdoub} implies that 
\begin{align}\label{equiv}
    \sup_{x \in T} \frac{m(x)}{m(p(x))}=:D_m<1.
\end{align}
 
\begin{lemma}\label{lemexp2} Assume that  $k>(D_m)^{p-1}$ for some $p\ge 1$. Then, 
\begin{align*}
    \sum_{x \in T_y} k^{\ell(x)}m(x)^p \approx k^{\ell(y)}m(y)^p.
\end{align*}
\end{lemma}
\begin{proof} The inequality $\ge$ is trivial. For the reverse one, observe that
\begin{align*}
    m(x)^{p-1} \le D_m^{(\ell(y)-\ell(x))(p-1)} m(y)^{p-1} \qquad \forall x \in T_y.
\end{align*} Then, 
\begin{align*}
     \sum_{x \in T_y} k^{\ell(x)}m(x)^p &\le m(y)^{p-1}D_m^{\ell(y)(p-1)}\sum_{x \in T_y}  \bigg(\frac{k}{D_m^{p-1}}\bigg)^{\ell(x)}m(x) \\ 
     &=m(y)^{p}D_m^{\ell(y)(p-1)}\sum_{n=0}^{\infty}  \bigg(\frac{k}{D_m^{p-1}}\bigg)^{(\ell(y)-n)} \\ 
     &\le Cm(y)^{p}k^{\ell(y)}.
\end{align*}
\end{proof}
\begin{remark}
    Notice that when $p=1$ in Lemma \ref{lemexp2}, the locally doubling assumption is not necessary. Similarly, if $k>1$, one can repeat the proof without any assumption on the flow measure $m$.
\end{remark}

 We now show that when $\sigma_c$ is an exponential function, it is possible to obtain a sharp estimate of the kernel of the Bergman projection. We point out that Assumption \hyperref[assumption:ii]{ii)} is fulfilled by $m\cdot \sigma_c$ when $m$ is locally doubling and $\sigma_c$ is an exponential function with base greater than $1$. 
\begin{proposition}
    Assume that there exists a constant $k_c > 1$ such that  
    $$  
    \sigma_c(n)= k_c^n, \quad \forall n \in \mathbb{Z}.  
    $$  
    Then, the Bergman kernel $K_c$ in $L^2(m \cdot \sigma_c)$ satisfies  
    \begin{equation}\label{estkc}
    |K_c(x,z)|\approx \frac{1}{\sigma_c(x \wedge z) m(x\wedge z)}, \quad \forall x,z \in T.
    \end{equation}
\end{proposition}
\begin{proof}
 
 Recall that $S_{n,x}(\sigma_c)=\sum_{k \le \ell(x)+n}\sigma_c(k)$ for every $n \in \mathbb Z$ and observe that $S_{n,x}(\sigma_c) \approx \sigma_c(\ell(x)+n)$.
By arguing much as in Remark \ref{rem:Psi} and using that $\sigma_c(n) \approx \sigma_c(n-1)$, we have the estimate 
$$|K_c(x,z)|\le C \frac{1}{\sigma_c(x\wedge z)m(x\wedge z)}.$$ For the reverse inequality notice that if $x$ and $z$ are one above the other then 
\begin{align*}
    |K_c(x,z)|&=\sum_{n=0}^\infty \frac{1}{S_{n,x\wedge z}(\sigma_c)}\bigg(\frac{1}{m(p^n(x\wedge z))}-\frac{1}{m(p^{n+1}(x\wedge z))}\bigg) \\ 
    &\ge  \frac{1}{S_{x\wedge z}(\sigma_c)}\bigg(\frac{1}{m(x\wedge z))}-\frac{1}{m(p(x\wedge z))}\bigg) \\ 
    &\ge \frac{1-D_m}{S_{x\wedge z}(\sigma_c)m(x\wedge z)}.
\end{align*}
 Otherwise, if $x$ and $z$ are not one above the other, recall that, since $m$ is locally doubling, $m(p^n(x)) \to \infty$ as $n \to \infty$ for every $x \in T$. Thus, we can estimate $|K_c(x, z)|$ from below by

\begin{align*} &\frac{1}{S_{-1,x\wedge z}(\sigma_c)m(x\wedge z)}-\sum_{n \ge 0} \frac{1}{S_{n,x\wedge z}(\sigma_c)}\bigg(\frac{1}{m(p^n(x\wedge z))}-\frac{1}{m(p^{n+1}(x\wedge z))}\bigg) \\ &\ge \frac{1}{S_{-1,x\wedge z}(\sigma_c)m(x\wedge z)}-\frac{1}{S_{x\wedge z}(\sigma_c)}\sum_{n \ge 0} \bigg(\frac{1}{m(p^n(x\wedge z))}-\frac{1}{m(p^{n+1}(x\wedge z))}\bigg)  \\ 
&= \frac{1}{S_{-1,x\wedge z}(\sigma_c)m(x\wedge z)}-\frac{1}{S_{x\wedge z}(\sigma_c)m(x\wedge z)},
\end{align*}
and
\begin{align*}
    \frac{1}{S_{-1,x\wedge z}(\sigma_c)}-\frac{1}{S_{x\wedge z}(\sigma_c)}=\frac{\sigma_c(x\wedge z)}{S_{-1,x\wedge z}(\sigma_c)S_{x\wedge z}(\sigma_c)} \approx \frac{1}{S_{-1,x\wedge z}(\sigma_c)} \ge \frac{1}{S_{x\wedge z}(\sigma_c)}. 
\end{align*}
Since $ S_{x\wedge z}(\sigma_c) \approx \sigma_c(x\wedge z)$, \eqref{estkc} follows.    
\end{proof}

From now on we assume that    $\sigma_{i}(n)=k_i^n$ for $i=a,b,c,d$, and that  $k_c>1$ so that sharp estimates of the Bergman kernel $K_c(\cdot,\cdot)$ are available. We shall prove necessary and sufficient conditions for the $L^p(m \cdot \sigma_d)$ boundedness of $\cU_{a,b,c}$ and $\cV_{a,b,c}$ and show that these conditions coincide in the particular case of a homogeneous tree. %

We begin by proving necessary conditions for the boundedness of $\cU_{a,b,c}$ (and therefore for $\cV_{a,b,c}$).
\begin{proposition}\label{prop:nec}
    Assume $\cU_{a,b,c}$ is bounded on $L^p(m \cdot \sigma_d)$ for some $p \in [1,\infty)$. Then, the following hold: 
    \begin{itemize}
        \item[\label{itm:one} i)] $m(x) \le C \bigg(\frac{k_c}{k_b k_a}\bigg)^{\ell(x)}$;
        \item[\label{itm:two} ii)]  $\frac{1}{k_aC_m}<k_d^{1/p}<k_b.$
    \end{itemize}
\end{proposition}
\begin{proof} Let $\delta_z$ denote the Dirac delta centered at $z \in T$.
Observe that
\begin{align}\label{testdelta}|\cU_{a,b,c}\delta_z(x)|=m(x)\sigma_a(x) |K_c(x,z)|m(z)\sigma_b(z).\end{align} 

Combining \eqref{estkc} and \eqref{testdelta}, we conclude that
\begin{align*}
    |\cU_{a,b,c} \delta_z(x)| \approx k_a^{\ell(x)}\frac{m(x)m(z) k_c^{-\ell(x\wedge z)} k_b^{\ell(z)}}{m(x\wedge z)}.
\end{align*}
If $\cU_{a,b,c}$ is bounded on $L^p(m\cdot \sigma_d)$, then 
\begin{align}\label{ineq}
    \sum_{x \in T} \bigg(k_a^{\ell(x)} \frac{m(x)m(z) k_c^{-\ell(x\wedge z)} k_b^{\ell(z)}}{m(x\wedge z)}\bigg)^p m(x) k_d^{\ell(x)} \le C m(z)k_d^{\ell(z)}.
\end{align}
By setting $T_{p^{-1}(z)}:=\emptyset$, we observe that the left-hand side in \eqref{ineq} is equal to
\begin{align}\label{rosso1}
&\sum_{n=0}^\infty \sum_{x \in T_{p^n(z)}\setminus T_{p^{n-1}(z)}} \bigg(k_a^{\ell(x)} \frac{m(x)m(z) k_c^{-(\ell(z)+n)} k_b^{\ell(z)}}{m(p^n(z))}\bigg)^p m(x) k_d^{\ell(x)}=:I.
\end{align} 

Since for every $x \in T_{p^n(z)}$
$$
m(x)^p\geq \left( \frac{m(p^n(z))}{C_m^{n+\ell(z)-\ell(x)}}\right)^p,
$$
we see that
\begin{align}\label{rosso2}
\sum_{x \in T_{p^n(z)}\setminus T_{p^{n-1}(z)}} \!\!\! \!\! \!\!\!\!\! (k_a^pk_d)^{\ell(x)}m(x)^{p+1} \geq \left( \frac{m(p^n(z))}{C_m^{n+\ell(z)}}\right)^p \sum_{x \in T_{p^n(z)}\setminus T_{p^{n-1}(z)}}\!\!\! \!\! \!\!\!\!\!(k_a^pk_dC_m^p)^{\ell(x)}m(x).
\end{align}
Let $s_{-1}(p^{n-1}(z)):=\emptyset$ and for every nonnegative integer $k$ set $$E_{k,n}:=\begin{cases}s_k(p^n(z))\setminus s_{k-1}(p^{n-1}(z)) &\text{if $n \ge 1$,} \\ 
s_k(p^n(z)) &\text{if $n=0$.}\end{cases}$$ Then,
\begin{align}\label{estbelowtent}
\sum_{x \in T_{p^n(z)}\setminus T_{p^{n-1}(z)}} (k_a^pk_d)^{\ell(x)}C_m^{\ell(x)p}m(x) &= \sum_{k=0}^\infty \sum_{x\in E_{k,n}}\left(C_m^pk_a^pk_d\right)^{\ell(z)+n-k}m(x)
\end{align} and the right-hand side is equal to 
\begin{align}
\nonumber&\left(C_m^pk_a^pk_d\right)^{\ell(z)+n}\sum_{k=0}^\infty \left(C_m^pk_a^pk_d\right)^{-k}\times\begin{cases}
    m(p^{n}(z))-m(p^{n-1}(z)) &\text{if $n \ge 1$}, \\ 
    m(z) &\text{if $n=0$.}
\end{cases} \\ 
\label{rosso3}&\approx \left(C_m^pk_a^pk_d\right)^{\ell(z)+n}m(p^n(z))\sum_{k=0}^\infty \left(C_m^pk_a^pk_d\right)^{-k}.
\end{align}

The last series is not convergent unless 
$$
k_d^{1/p} >1/(C_mk_a),
$$
 which is the left inequality in   \hyperref[itm:two]{ii)}. Furthermore, putting together \eqref{rosso1}, \eqref{rosso2}, and \eqref{rosso3}, we get
\begin{align*}
    &I \ge  m(z)^p k_d^{\ell(z)}\bigg(\frac{k_bk_a}{k_c}\bigg)^{p\ell(z)}\sum_{n=0}^\infty \frac{k_c^{-np}k_a^{np}k_d^{n}}{m(p^n(z))^p} m(p^{n}(z))^{p+1} \\ 
&=m(z)^p k_d^{\ell(z)}\bigg(\frac{k_bk_a}{k_c}\bigg)^{p\ell(z)}\sum_{n=0}^\infty k_c^{-np}k_a^{np}k_d^{n}m(p^{n}(z)) \\ 
&=:II
\end{align*}

It is clear that
\begin{align*} 
II&\ge m(z)^{p+1}k_d^{\ell(z)}\bigg(\frac{k_bk_a}{k_c}\bigg)^{p\ell(z)}.
\end{align*}
It follows by \eqref{ineq} that
\begin{align*}{m(z)^{p+1}}k_d^{\ell(z)}\bigg(\frac{k_bk_a}{k_c}\bigg)^{p\ell(z)} &\le
Cm(z)k_d^{\ell(z)},
\end{align*}
 which is equivalent to \hyperref[itm:one]{i)}.

It remains to prove the second inequality in \hyperref[itm:two]{ii)}. It is easy to see that the adjoint operator of $\cU_{a,b,c}$ on $L^2(m \cdot \sigma_d)$ is defined by $$\cU_{a,b,c}^*f(x)=\bigg(\frac{k_b}{k_d}\bigg)^{\ell(x)}\sum_{y \in T}K_c(x,y)(k_ak_d)^{\ell(y)}f(y)m^2(y) \qquad \forall f \in L^2(m \cdot \sigma_d).$$  
It follows that for every $p \in [1,\infty)$ the operator $\cU_{a,b,c}$ is bounded on $L^p(m \cdot \sigma_d)$  only if $\cU_{a,b,c}^*$ is bounded on $L^{p'}(m \cdot \sigma_d)$ where $1/p+1/p'=1.$

By testing $\cU^*_{a,b,c}$ on $\delta_z$, we get that $\cU_{a,b,c}$ is bounded on $L^{p}(m\cdot \sigma_d)$ for $p \in [1,\infty)$ only if there exists a constant $C>0$ such that
\begin{align}\label{dual}\bigg[\sum_{x \in T}\bigg(\big(k_b/k_d\big)^{\ell(x)}|K_c(x,z)|(k_ak_d)^{\ell(z)}m^2(z)\bigg)^{p'}m(x) k_d^{\ell(x)}\bigg]^{1/p'} \le C \big(m(z)k_d^{\ell(z)}\big)^{1/p'}\end{align}
 with the obvious modification when $p=1$. Assume $p \in (1,\infty)$.
Notice that the left-hand side in \eqref{dual} is  comparable to
\begin{align*}
    \big(k_ak_d/k_c\big)^{\ell(z)}m^2(z)\bigg[\sum_{n=0}^\infty\frac{1}{k_c^{np'}} \sum_{x \in T_{p^n}(z) \setminus T_{p^{n-1}(z)}} \big(k_b/k_d)^{p'\ell(x)}k_d^{\ell(x)}m(x)\bigg]^{1/p'}.
\end{align*} By repeating the computations performed in \eqref{estbelowtent}, it is easy to see that the inner sum above is not convergent unless $(k_b/k_d)^{p'}k_d>1$, that is equivalent to $$k_b>k_d^{1/p},$$ providing the desired inequality. The case $p=1$ can be obtained by a similar and easier argument, so we omit the details.

\end{proof}

In the next proposition we collect sufficient conditions for the boundedness of $\cV_{a,b,c}$ (and hence for $\cU_{a,b,c})$.

\begin{proposition}\label{prop:suff} Let $p \in (1,\infty)$ and
    assume that 
    \begin{itemize}
        \item[\label{item:1}1)] $m(x) \le c \big(\frac{k_c}{k_bk_a}\big)^{\ell(x)}$ for every $x \in T$;
       % \item[2)]$k_c>(C_m D_m)^{3-1/p'}$;
        \item[\label{item:2}2)] $\frac{C_mD_m^{2-1/p'}k_b^{1/p'}}{k_a^{1/p}k_c^{1/p'}}\label{kd}<k_d^{1/p}<\frac{k_b^{1/p'}k_c^{1/p}}{C_m^{2-1/p'}D_m k_a^{1/p}}.$ 
    \end{itemize}
     Then, $\cV_{a,b,c}$ is bounded on $L^p(m \cdot \sigma_d)$.
\end{proposition}
\begin{proof}We shall apply Schur's Theorem. Fix $h(x)=k_\gamma^{ \ell(x)}m(x)$ for some $k_\gamma$ satisfying 
\begin{align} 
   \frac{D_m^{p'}}{k_b} < &k_\gamma^{p'}< \frac{k_c}{k_b C_m^{p'}}, \label{kga1}\\
    \frac{D_m^{p+1}}{k_ak_d} < &k_\gamma^{p}< \frac{k_c}{k_ak_d C_m^{p+1}} \label{kga2}.
\end{align} It is not hard to see that a $k_\gamma$ satisfying the above inequalities exists assuming \hyperref[item:2]{2)} (see also Remark \ref{rem:final}). Set $$H(x,y):=m(x)k_a^{\ell(x)}|K_c(x,y)|k_b^{\ell(y)}k_d^{-\ell(y)},$$ so that $\cV_{a,b,c} f(x)=\sum_{y \in T} H(x,y) f(y) m(y) k_d^{\ell(y)}$. We have that
\begin{align*}
    \sum_{y \in T} H(x,y) h(y)^{p'} m(y) k_d^{\ell(y)}&=\sum_{y \in T} k_a^{\ell(x)}m(x)|K_c(x,y)|k_b^{\ell(y)}k_\gamma^{p' \ell(y)}m(y)^{p'+1} \\ 
    &\approx m(x)k_a^{\ell(x)}\sum_{y \in T} \frac{k_b^{\ell(y)}k_\gamma^{p' \ell(y)}}{m(x\wedge y)k_c^{\ell(x\wedge y)}}m(y)^{p'+1}   \\ 
    &= m(x)k_a^{\ell(x)}\sum_{n=0}^\infty \sum_{y \in T_{p^n(x)}\setminus T_{p^{n-1}(x)}} \frac{k_b^{\ell(y)}k_\gamma^{p' \ell(y)}m(y)^{p'+1}}{m(p^n(x))k_c^{\ell(x)+n}} \\ 
    &\le  m(x)\bigg(\frac{k_a}{k_c}\bigg)^{\ell(x)} \sum_{n=0}^\infty \sum_{y \in T_{p^n(x)}} \frac{(k_bk_\gamma^{p'})^{\ell(y)}}{{m(p^n(x))k_c^n}}m(y)^{p'+1}\\ 
    &\le C  m(x)\bigg(\frac{k_a k_b k_\gamma^{p'}}{k_c}\bigg)^{\ell(x)} \sum_{n=0}^\infty \frac{k_b^{n}k_\gamma^{np'}m(p^n(x))^{p'}}{k_c^n} \\ 
    &\le C \bigg(\frac{k_a k_b k_\gamma^{p'}}{k_c}\bigg)^{\ell(x)} m(x)^{p'+1}\\ 
    &\le C k_\gamma^{p' \ell(x)} m(x)^{p'} \\ 
    &=C h(x)^{p'},
\end{align*} where we have used Lemma \ref{lemexp2} and \eqref{kga1} in the third-to-last inequality, and again \eqref{kga1} and 
$$m(p^n(x))^{p'} \le C_m^{np'}m(x)^{p'}$$  in the penultimate inequality. Finally, the last inequality follows from \hyperref[item:1]{1)}.
 Moreover, 
\begin{align*}
    \sum_{x \in T} H(x,y) h(x)^{p} k_d^{\ell(x)}m(x) &\le \sum_{x \in T} k_a^{\ell(x)}|K_c(x, y)| \bigg(\frac{k_b}{k_d}\bigg)^{\ell(y)}k_{\gamma}^{p\ell(x)} k_d^{\ell(x)}m(x)^{p+2} \\ 
    &=\bigg(\frac{k_b}{k_d}\bigg)^{\ell(y)}\sum_{x \in T} \frac{k_a^{\ell(x)}}{k_c^{\ell(x\wedge y)}m(x\wedge y)}k_\gamma^{p\ell(x)}k_d^{\ell(x)}m(x)^{p+2} \\ 
        &=\bigg(\frac{k_b}{k_d}\bigg)^{\ell(y)}\sum_{n=0}^\infty\sum_{x \in T_{p^n(y)}\setminus T_{p^{n-1}(y)}}\!\!\!\frac{k_a^{\ell(x)}k_\gamma^{p\ell(x)}k_d^{\ell(x)}}{k_c^{\ell(y)+n}m(p^n(y))}m(x)^{p+2}\\ 
        &\le \bigg(\frac{k_b}{k_dk_c}\bigg)^{\ell(y)}\sum_{n=0}^\infty\frac{1}{k_c^n m(p^n(y))}\!\sum_{x \in T_{p^n(y)}}\!\!\!\!\!(k_ak_\gamma^{p}k_d)^{\ell(x)}m(x)^{p+2} \\ 
        &\le C \bigg(\frac{k_ak_bk_\gamma^{p}}{k_c }\bigg)^{\ell(y)}\sum_{n=0}^\infty\frac{1}{k_c^n}k_a^{n}k_\gamma^{pn}k_d^{n}m(p^n(y))^{p+1} \\ 
        &\le C\bigg(\frac{k_ak_bk_\gamma^{p}}{k_c }\bigg)^{\ell(y)}m(y)^{p+1},
\end{align*}where the second to last inequality follows by using Lemma \ref{lemexp2} and that ${k_ak_\gamma^{p}k_d>D_m^{p+1}}$ and the last inequality follows by
$$k_ak_\gamma^{p}k_dC_m^{p+1}<k_c.$$ We conclude by observing that \hyperref[item:1]{1)} implies that for every $y \in T$ $$m(y) \bigg(\frac{k_ak_b}{k_c}\bigg)^{\ell(y)} \le C,$$ so that $$\sum_{x \in T} H(x,y) h(x)^{p} k_d^{\ell(x)}m(x) \le C h(y)^p.$$ 
\end{proof}

The case $p=1$ can be obtained by a similar and easier analysis. 
\begin{proposition}\label{PROP:P=1}
    Assume that $m(x) \le C \big(k_c/(k_ak_b)\big)^{\ell(x)}$ for every $x \in T$ and that $D_m<k_a k_d<k_c/C_m $. Then, $\cV_{a,b,c}$ is bounded on $L^1(m \cdot \sigma_d).$
\end{proposition}
\begin{proof}
    It is clear that $$\|\cV_{a,b,c}\|_{L^1(m\cdot \sigma_d);L^1(m\cdot \sigma_d)}\le \sup_{z \in T}\bigg(\frac{k_b}{k_d}\bigg)^{\ell(z)} \sum_{x \in T}m(x)k_a^{\ell(x)}|K_c(x,z)|m(x)\sigma_d(x)=:I.$$
    By \eqref{estkc}
    \begin{align*}
        I &\approx \sup_{z \in T} \bigg(\frac{k_b}{k_d}\bigg)^{\ell(z)} \sum_{x \in T}\frac{m^2(x)(k_ak_d)^{\ell(x)}}{k_c^{\ell(x \wedge z)}m(x \wedge z)} \\ 
        &=\sup_{z \in T} \bigg(\frac{k_b}{k_dk_c}\bigg)^{\ell(z)} \sum_{n=0}^{\infty}\frac{1}{k_c^n{m(p^n(z))} }\sum_{x \in T_{p^{n}(z)} \setminus T_{p^{n-1}(z)}} m(x)^2(k_ak_d)^{\ell(x)}.
    \end{align*} An application of Lemma \ref{lemexp2} now yields
    \begin{align*}
     I &\le C \sup_{z \in T} \bigg(\frac{k_bk_a}{k_c}\bigg)^{\ell(z)} \sum_{n=0}^\infty \frac{1}{k_c^n} m(p^n(z)) (k_a k_d)^n   \\ 
     &\le C\sup_{z \in T} m(z)\bigg(\frac{k_bk_a}{k_c}\bigg)^{\ell(z)} \sum_{n=0}^\infty   \bigg(C_m\frac{k_a k_d}{k_c}\bigg)^n \\ 
     &\le C.
    \end{align*}
\end{proof} 
\begin{remark}\label{rem:final}

 It is important to verify that the two inequalities \eqref{kga1} and \eqref{kga2} in Proposition \ref{prop:suff} can be satisfied. We can rewrite them as
\begin{align*}
     \frac{D_m}{k_b^{1/p}} < &k_\gamma< \left(\frac{k_c}{k_b}\right)^{1/p}\frac{1}{C_m}, \\
    \frac{D_m^{1+1/p'}}{(k_ak_d)^{1/p'}} < &k_\gamma< \left(\frac{k_c}{k_ak_d}\right)^{1/p'}\frac{1}{C_m^{1+1/p'}}.
\end{align*}
Such $k_\gamma$ exists if the two intervals defined by the two pairs of conditions above intersect, namely if
\begin{align*}
  &\frac{D_m^{2-1/p'}}{(k_ak_d)^{1/p}}<  \left(\frac{k_c}{k_b}\right)^{1/p'}\frac{1}{C_m}\\
  &\frac{D_m}{k_b^{1/p'}} <  \left(\frac{k_c}{k_ak_d}\right)^{1/p}\frac{1}{C_m^{2-1/p'}},
\end{align*}
where we substituted $1/p'=1-1/p$. By solving the inequalities for $k_d$ we obtain
\begin{align}
    k_d^{1/p}&>\frac{C_mD_m^{2-1/p'}k_b^{1/p'}}{k_a^{1/p}k_c^{1/p'}}\label{kd1}\\
    k_d^{1/p}&<\frac{k_b^{1/p'}k_c^{1/p}}{C_m^{2-1/p'}D_m k_a^{1/p}},\label{kd2}
\end{align}  which is exactly \hyperref[item:2]{2)}.
It is then straightforward to conclude that $k_d$ exists if and only if
\begin{equation}\label{intersection-not-empty}
    k_c> (C_mD_m)^{3-1/p'}.
\end{equation}
\end{remark}
Finally, we prove that for homogeneous trees equipped with the canonical flow measure, necessary and sufficient conditions for the $L^p$ boundedness of $\cU_{a,b,c}$ and $\cV_{a,b,c}$ coincide.
\begin{corollary}\label{equivalence-Toeplitz}
Let $T_{q+1}$ be the homogeneous tree of order $q+1$ and let $m(\cdot)=q^{\ell(\cdot)}$ be the canonical flow measure on $T_{q+1}$. Set $\sigma_j(\cdot)=k_j^{\ell(\cdot)}=q^{(j-1)\ell(\cdot)}$  for $j \in \{a,b,c,d\}$, with $c>1$ and let $p \in [1,\infty)$. The following are equivalent: 
\begin{itemize}
    \item[i')] $\cU_{a,b,c}$ is bounded on $L^{p}(m \cdot \sigma_d)$;
    \item[ii')] $\cV_{a,b,c}$ is bounded on $L^{p}(m \cdot \sigma_d)$;
    \item[iii')] $-a<\frac{d-1}{p}<b-1$ and  $c+1=a+b$. 
\end{itemize}

\end{corollary}
\begin{proof}
    The inequality $m(\cdot) \le C \bigg(\frac{k_c}{k_ak_b}\bigg)^{\ell(\cdot)}$ on $T_{q+1}$ holds if and only if $\frac{k_c}{k_ak_b}=q$, that is the equality in iii'). Since in this case $C_m=q=1/D_m$, it is not hard to see that \hyperref[itm:two]{ii)} and \hyperref[item:2]{2)} (and the corresponding inequality when $p=1$ in Proposition \ref{PROP:P=1}) are equivalent to the inequalities in iii').
\end{proof}

 \bibliographystyle{plain}
{\small
\bibliography{references}}
\end{document}